\documentclass[a4paper,12pt]{amsart}
\usepackage{amssymb}
\usepackage{pifont} 
\usepackage{bbding}
\usepackage{wasysym}
\usepackage{ifthen}
\usepackage{cite}
 \usepackage[dvips]{graphicx}
\nonstopmode \numberwithin{equation}{section}
\usepackage{amssymb}
\usepackage{ifthen}
\usepackage{graphicx}
\usepackage{amsmath}
\usepackage[T1]{fontenc} 
\usepackage[utf8]{inputenc}
\usepackage[usenames,dvipsnames]{color}
\usepackage{color}
\usepackage[english]{babel}
\usepackage{fancyhdr}
\usepackage{fancybox}
\usepackage{tikz}

\nonstopmode \numberwithin{equation}{section}
\theoremstyle{plain}

\newtheorem{conj}{Conjecture}

\theoremstyle{definition}
\newtheorem{defi}{Definition}[section]

\newtheorem{thm}{Theorem}[section]
\newtheorem{prob}{Problem}[section]
\newtheorem{cor}{Corollary}[section]

\newtheorem{prop}{Proposition}[section]
\newtheorem{rem}{Remark}[section]
\newtheorem{lem}{Lemma}[section]

\newcounter{minutes}
\divide\time by 60
\newcounter{hours}
\multiply\time by 60
\addtocounter{minutes}{-\time}

\newcounter {own}
\def\theown {\thesection       .\arabic{own}}

\newenvironment{pf}[1][]{%
 \vskip 3mm
 \noindent
 \ifthenelse{\equal{#1}{}}%
  {{\slshape Proof. }}%
  {{\slshape #1.} }%
 }%
{\qed\bigskip}

\newcounter{alphabet}

\def\be{\begin{equation}}
\def\ee{\end{equation}}

\newcommand{\bee}{\begin{enumerate}}
\newcommand{\eee}{\end{enumerate}}

\newcommand{\blem}{\begin{lem}}
\newcommand{\elem}{\end{lem}}
\newcommand{\bthm}{\begin{thm}}
\newcommand{\ethm}{\end{thm}}
\newcommand{\bcor}{\begin{cor}}
\newcommand{\ecor}{\end{cor}}
\newcommand{\beg}{\begin{examp}}
\newcommand{\eeg}{\end{examp}}
\newcommand{\begs}{\begin{examples}}
\newcommand{\eegs}{\end{examples}}

\newcommand{\bdefn}{\begin{defn}}
\newcommand{\edefn}{\end{defn}}

\newcommand{\bprob}{\begin{prob}}
\newcommand{\eprob}{\end{prob}}
\newcommand{\bei}{\begin{itemize}}
\newcommand{\eei}{\end{itemize}}

\newcommand{\bcon}{\begin{conj}}
\newcommand{\econ}{\end{conj}}
\newcommand{\bcons}{\begin{conjs}}
\newcommand{\econs}{\end{conjs}}
\newcommand{\bprop}{\begin{prop}}
\newcommand{\eprop}{\end{prop}}
\newcommand{\br}{\begin{rem}}
\newcommand{\er}{\end{rem}}
\newcommand{\brs}{\begin{rems}}
\newcommand{\ers}{\end{rems}}
\newcommand{\bo}{\begin{obser}}
\newcommand{\eo}{\end{obser}}
\newcommand{\bos}{\begin{obsers}}
\newcommand{\eos}{\end{obsers}}
\newcommand{\bpf}{\begin{pf}}
\newcommand{\epf}{\end{pf}}
\newcommand{\ba}{\begin{array}}
\newcommand{\ea}{\end{array}}
\newcommand{\beq}{\begin{eqnarray}}
\newcommand{\beqq}{\begin{eqnarray*}}
\newcommand{\eeq}{\end{eqnarray}}
\newcommand{\eeqq}{\end{eqnarray*}}

\begin{document}

\title{Second Hankel determinant of Logarithmic coefficients for Starlike and Convex functions associated with lune}
\author{Sanju Mandal}
\address{Sanju Mandal, Department of Mathematics, Jadavpur University, Kolkata-700032, West Bengal,India.}
\email{sanjum.math.rs@jadavpuruniversity.in}

\author{Molla Basir Ahamed}
\address{Molla Basir Ahamed, Department of Mathematics, Jadavpur University, Kolkata-700032, West Bengal,India.}
\email{mbahamed.math@jadavpuruniversity.in}

\subjclass[{AMS} Subject Classification:]{Primary 30A10, 30H05, 30C35, Secondary 30C45}

\def\thefootnote{}
\footnotetext{ {\tiny File:~\jobname.tex,
printed: \number\year-\number\month-\number\day,
          \thehours.\ifnum\theminutes<10{0}\fi\theminutes }
} \makeatletter\def\thefootnote{\@arabic\c@footnote}\makeatother
\begin{abstract} 
The Hankel determinant $H_{2,1}(F_{f}/2)$ is defined as:
\begin{align*}
	H_{2,1}(F_{f}/2):= \begin{vmatrix}
		\gamma_1 & \gamma_2 \\
		\gamma_2 & \gamma_3
	\end{vmatrix},
\end{align*}
where $\gamma_1, \gamma_2,$ and $\gamma_3$ are the first, second and third logarithmic coefficients of functions belonging to the class $\mathcal{S}$ of normalized univalent functions. In this article, we establish sharp inequalities $|H_{2,1}(F_{f}/2)|\leq 1/16$ and $|H_{2,1}(F_{f}/2)| \leq 23/3264$ for the logarithmic coefficients of starlike and convex functions associated with lune.
\end{abstract}
\maketitle
\pagestyle{myheadings}
\markboth{S. Mandal and M. B. Ahamed}{Second Hankel determinant of Logarithmic coefficients}

\section{Introduction}
Suppose $\mathcal{H}$ be the class of functions $ f $ which are holomorphic in the open unit disk $\mathbb{D}=\{z\in\mathbb{C}: |z|<1\}$ of the form 
\begin{align}\label{eq-1.1}
	f(z)=\sum_{n=1}^{\infty}a_nz^n,\; \mbox{for}\; z\in\mathbb{D}.
\end{align}
Then $\mathcal{H}$ is a locally convex topological vector space endowed with the topology of uniform convergence over compact subsets of $\mathbb{D}$. Let $\mathcal{A}$ denote the class of functions $f\in\mathcal{H}$ such that $f(0)=0$ and $f^{\prime}(0)=1$. That is, the functions $f$ of the form
\begin{align}\label{eq-1.2}
	f(z)=z+ \sum_{n=2}^{\infty}a_nz^n,\; \mbox{for}\; z\in\mathbb{D}.
\end{align} 
Let $\mathcal{S}$ denote the subclass of all functions in $\mathcal{A}$ which are univalent. For a general theory of univalent functions, we refer the classical books \cite{Duren-1983-NY,Goodman-1983}.

Let
\begin{align}\label{eq-1.3}
	F_{f}(z):=\log\dfrac{f(z)}{z}=2\sum_{n=1}^{\infty}\gamma_{n}(f)z^n, \;\; z\in\mathbb{D},\;\;\log 1:=0,
\end{align}
be a logarithmic function associated with $f\in\mathcal{S}$. The numbers $\gamma_{n}:=\gamma_{n}(f)$ are called the logarithmic coefficients of $f$. Although the logarithmic coefficients $\gamma_{n}$ play a critical role in the theory of univalent functions, it appears that there are only a limited number of exact upper bounds established for them. As is well known, the logarithmic coefficients play a crucial role in Milin’s conjecture (\cite{Milin-1977-ET}, see also\cite[p.155]{Duren-1983-NY}). We note that for the class $\mathcal{S}$ sharp estimates are known only for $\gamma_{1}$ and $\gamma_{2}$, namely
\begin{align*}
	|\gamma_{1}|\leq 1, \;\; |\gamma_{2}|\leq \dfrac{1}{2}+ \dfrac{1}{e} =0.635\ldots
\end{align*}
Estimating the modulus of logarithmic coefficients for $f\in\mathcal{S}$ and various sub-classes has been considered recently by several authors. We refer to the articles \cite{Ali-Allu-PAMS-2018, Ali-Allu-Thomas-CRMCS-2018,Cho-Kowalczyk-kwon-Lecko-Sim-RACSAM-2020,Girela-AASF-2000,Thomas-PAMS-2016} and references therein.\vspace{2mm}

\begin{defi}\label{def-1.1}
	Let $f$ and $g$ be two analytic functions. Then $f$ is subordinated by $g$ and written as $f(z)\prec g(z)$, if there exists a self map $w$ such that $w(0)=0$ such that $f(z)=g(w(z))$. Moreover, if $g$ is univalent and $f(0)=g(0)$, then $f(\mathbb{D})\subseteq g(\mathbb{D})$.
\end{defi} 
Raina and Sokol \cite{Raina-Sokół-CRMASCP-2015} introduced the class $\mathcal{S}^{*}_{\mbox{\leftmoon}}$ given by
\begin{align*}
	\mathcal{S}^{*}_{\mbox{\leftmoon}}:=\left\{f\in\mathcal{S}:\; \vline\left(\frac{zf^{\prime}(z)}{f(z)}\right)^2 -1\vline\leq 2\vline\frac{zf^{\prime}(z)}{f(z)}\vline,\;\; z\in\mathbb{D}\right\}.
\end{align*}
Geometrically, a function $f\in\mathcal{S}^{*}_{\mbox{\leftmoon}}$ is that, for any $z\in\mathbb{D}$, the ratio $\frac{zf^{\prime}(z)}{f(z)}$ contains the region which is bounded by the lune. It is given by the relation $\left\{w\in\mathbb{C}:|w^2 -1|\leq 2|w|\right\}$. By using the definition of the subordination, the class $\mathcal{S}^{*}_{\mbox{\leftmoon}}$ is defined as
\begin{align*}
	\mathcal{S}^{*}_{\mbox{\leftmoon}}:=\left\{f\in\mathcal{S}: \frac{zf^{\prime}(z)}{f(z)}\prec z+\sqrt{1+ z^2}=q(z),\;\; z\in\mathbb{D}\right\},
\end{align*}
where branch of the square root is chosen to be $q(0)=1$. The class $\mathcal{C}^{*}_{\mbox{\leftmoon}}$ of convex function $q$ is defined as
\begin{align*}
	\mathcal{C}^{*}_{\mbox{\leftmoon}}:=\left\{f\in\mathcal{S}: 1+ \frac{ zf^{\prime\prime}(z)}{f^{\prime}(z)}\prec q(z),\;\; z\in\mathbb{D}\right\},
\end{align*}
The class $\mathcal{S}^{*}_{\mbox{\leftmoon}}$ has been the subject of extensive investigation by several authors. The coefficient estimates of the class $\mathcal{S}^{*}_{\mbox{\leftmoon}}$ were investigated by Raina and Sokol \cite{Raina-Sokół-HJMS-2015, Raina-Sokół-MMN-2015}, whereas Gandhi and Ravichandran \cite{Gandhi-Ravichandran-AEJM-2017} examined the radius issues associated with the same class. Certain differential subordinations related to the class $\mathcal{S}^{*}_{\mbox{\leftmoon}}$ were studied by Sharma \textit{et al.} \cite{Sharma-Raina-Sokół-AMP-2019}.
Raina \textit{et al.} \cite{Raina-Sharma-Sokół-JCMA-2018} give integral representation and sufficient conditions for the functions in the class $\mathcal{S}^{*}_{\mbox{\leftmoon}}$. A recent contribution by Cho \textit{et al.} \cite{Cho-Kumar-Kwon-Sim-JIA-2019} proposed a conjecture regarding the coefficients of this particular class. \vspace{2mm}

In geometric function theory, a lot of emphasis have been given to evaluate the bounds of Hankel determinants, whose elements are the coefficients of analytic functions $f$ characterize in $\mathbb{D}$ of the form \eqref{eq-1.2}. Hankel matrices (and determinants) play a key role in several branches of mathematics and have various applications \cite{Ye-Lim-FCM-2016}.\vspace{1.2mm}

This study is dedicated to providing the sharp bound for the second Hankel determinant, whose entries are the logarithmic coefficients. We commence by presenting the definitions of Hankel determinants in the case where $f\in \mathcal{A}$.\vspace{1.2mm}

The Hankel determinant $H_{q,n}(f)$ of Taylor's coefficients of functions $f\in\mathcal{A}$ represented by \eqref{eq-1.1} is defined for $q,n\in\mathbb{N}$ as follows:
\begin{align*}
	H_{q,n}(f):=\begin{vmatrix}
		a_{n} & a_{n+1} &\cdots& a_{n+q-1}\\ a_{n+1} & a_{n+2} &\cdots& a_{n+q} \\ \vdots & \vdots & \vdots & \vdots \\ a_{n+q-1} & a_{n+q} &\cdots& a_{n+2(q-1)}
	\end{vmatrix}.
\end{align*}
Kowalczyk and Lecko \cite{Kowalczyk-Lecko-BAMS-2022} recently proposed a Hankel determinant whose elements are the logarithmic coefficients of $f\in\mathcal{S}$, realizing the extensive use of these coefficients. This determinant is expressed as follows:
\begin{align*}
	H_{q,n}(F_{f}/2)=\begin{vmatrix}
		\gamma_{n} & \gamma_{n+1} &\cdots& \gamma_{n+q-1}\\ \gamma_{n+1} & \gamma_{n+2} &\cdots& \gamma_{n+q} \\ \vdots & \vdots & \vdots & \vdots \\ \gamma_{n+q-1} & \gamma_{n+q} &\cdots& \gamma_{n+2(q-1)}
	\end{vmatrix}.
\end{align*}

The study of Hankel determinants for starlike, convex, or many other functions has been done extensively (see \cite{Ponnusamy-Sugawa-BDSM-2021,Kowalczyk-Lecko-RACSAM-2023,Raza-Riza-Thomas-BAMS-2023,Sim-Lecko-Thomas-AMPA-2021,Kowalczyk-Lecko-BAMS-2022}), their sharp bounds have been established. Recently, the Hankel determinants with logarithmic coefficients have been examined for certain subclasses of starlike, convex, univalent, strongly starlike and strongly convex functions (see \cite{Allu-Arora-Shaji-MJM-2023,Kowalczyk-Lecko-BAMS-2022,Kowalczyk-Lecko-BMMS-2022} and references therein). However, a little is known about sharp bounds of Hankel determinants of logarithmic coefficients and need to explore them for many classes of functions.\vspace{2mm} 

Differentiating \eqref{eq-1.3} and using \eqref{eq-1.2}, a simple computation shows that 
\begin{align*}
	\begin{cases}
		\gamma_{1}=\dfrac{1}{2}a_{2},\vspace{1.5mm}\\ \gamma_{2}=\dfrac{1}{2} \left(a_{3} -\dfrac{1}{2}a^2_{2}\right), \vspace{1.5mm}\\ \gamma_{3} =\dfrac{1}{2}\left(a_{4}- a_{2}a_{3} +\dfrac{1}{3}a^3_{2}\right), \vspace{1.5mm}\\ \gamma_{4}= \dfrac{1}{2} \left(a_{5} -a_{2}a_{4} +a^2_{2} a_{3} -\dfrac{1}{2}a^2_{3} -\dfrac{1}{4}a^4_{2}\right), \vspace{1.5mm}\\ \gamma_{5}= \dfrac{1}{2}\left(a_{6} -a_{2}a_{5} -a_{3}a_{4} +a_{2} a^2_{3} + a^2_{2} a_{4} -a^3_{2} a_{3} +\dfrac{1}{5}a^5_{2}\right).
	\end{cases}
\end{align*}

Due to the great importance of logarithmic coefficients in the recent years, it is appropriate and interesting to compute the Hankel determinant whose entries are logarithmic coefficients. In particular, the second Hankel determinant of $F_{f}/2$ is defined as
\begin{align}\label{eq-1.4}
	H_{2,1}(F_{f}/2)=\gamma_{1}\gamma_{3} -\gamma^2_{2}=\dfrac{1}{48} \left(a^4_2 - 12 a^2_3 + 12 a_2 a_4\right).
\end{align}
In this paper, we aim to explore by examining the sharp bound of the Hankel determinant $H_{2,1}(F_{f}/2)$ for two class of functions, namely, starlike and convex functions associated with lune.\vspace{2mm}

It is known that for the Koebe function $f(z)=z/(1-z)^2$, the logarithmic coefficients are $\gamma_{n}=1/n$, for each positive integer $n$. Since the Koebe function appears as an extremal function in many problems of geometric theory of analytic functions, one could expect that $\gamma_{n}=1/n$ holds for functions in $\mathcal{S}$. But this is not true in general, even in order of magnitude. The problem of computing
the bound of the logarithmic coefficients are studied
recently by several authors in different contexts, for instance see \cite{Ali-Allu-PAMS-2018,Ali-Allu-Thomas-CRMCS-2018,Cho-Kowalczyk-kwon-Lecko-Sim-RACSAM-2020,Thomas-PAMS-2016,Ponnusamy-Sugawa-BDSM-2021}. \vspace{2mm}

As usual, instead of the whole class $\mathcal{S}$, one can take into account their subclasses for which the problem of finding sharp estimates of Hankel determinant of logarithmic coefficients can be studied. The problem of computing the sharp bounds of $H_{2,1}(F_{f}/2)$ was considered
in \cite{Kowalczyk-Lecko-BAMS-2022} for starlike and convex functions. It is now appropriate to remark that $H_{2,1}(F_{f}/2)$ is invariant under rotation since for $f_{\theta}(z):=e^{-i\theta}f(e^{i\theta}z)$, $\theta\in\mathbb{R}$ when $f\in\mathcal{S}$ we have
\begin{align*}
	H_{2,1}(F_{f_{\theta}}/2)=\frac{e^{4i\theta}}{48}\left(a^4_2 - 12 a^2_3 + 12 a_2 a_4\right)=e^{4i\theta}H_{2,1}(F_{f}/2).
\end{align*}

\section{Preliminary results}
The Carath$\acute{e}$odory class $\mathcal{P}$ and its coefficients bounds plays a significant roles in establishing the bounds of Hankel determinants. The class $\mathcal{P}$ of analytic functions $h$ defined for $z\in\mathbb{D}$ is given by
\begin{align}\label{eq-1.5}
	p(z)=1+\sum_{n=1}^{\infty}c_n z^n
\end{align}
with positive real part in $\mathbb{D}$. A member of $\mathcal{P}$ is called a Carath$\acute{e}$odory function. It is known that $c_n\leq 2$, $n\geq 1$ for a function $p\in\mathcal{P}$ (see \cite{Duren-1983-NY}).\vspace{2mm}

In this section, we present key lemmas which will be used to prove the main results of this paper. Parametric representations of the coefficients are often useful in finding the bound of Hankel determinants, in this regard, Libera and Zlotkiewicz \cite{Libera-Zlotkiewicz-PAMS-1982, Libera-Zlotkiewicz-PAMS-1983} derived the following parameterizations of possible values of $c_2$ and $c_3$.
\begin{lem}\cite{Libera-Zlotkiewicz-PAMS-1982,Libera-Zlotkiewicz-PAMS-1983}\label{lem-2.1}
If $p\in\mathcal{P}$ is of the form \eqref{eq-1.5} with $c_1\geq 0$, then 
\begin{align}\label{eq-2.1}
	&c_1=2\tau_1,\\\label{eq-2.2} &c_2=2\tau^2_1 +2(1-\tau^2_1)\tau_2
\end{align}
and
\begin{align}\label{eq-2.3}
	c_3= 2\tau^3_1  + 4(1 -\tau^2_1)\tau_1\tau_2 - 2(1 - \tau^2_1)\tau_1\tau^2_2 + 
	2(1 - \tau^2_1)(1 - |\tau_2|^2)\tau_3
\end{align}
for some $\tau_1\in[0,1]$ and $\tau_2,\tau_3\in\overline{\mathbb{D}}:= \{z\in\mathbb{C}:|z|\leq 1\}$.\vspace{1.2mm}

For $\tau_1\in\mathbb{T}:=\{z\in\mathbb{C}:|z|=1\}$, there is a unique function $p\in\mathcal{P}$ with $c_1$ as in \eqref{eq-2.1}, namely
\begin{align*}
	p(z)=\frac{1+\tau_1 z}{1-\tau_1 z}, \;\;z\in\mathbb{D}.
\end{align*}

For $\tau_1\in\mathbb{D}$ and $\tau_2\in\mathbb{T}$, there is a unique function $p\in\mathcal{P}$ with $c_1$ and $c_2$ as in \eqref{eq-2.1} and \eqref{eq-2.2}, namely
\begin{align*}
	p(z)=\frac{1+(\overline{\tau_1}\tau_2 +\tau_1)z+\tau_2 z^2}{1 +(\overline{\tau_1}\tau_2 -\tau_1)z-\tau_2 z^2}, \;\;z\in\mathbb{D}.
\end{align*}

For $\tau_1,\tau_2\in\mathbb{D}$ and $\tau_3\in\mathbb{T}$, there is a unique function $p\in\mathcal{P}$ with $c_1,c_2$ and $c_3$ as in \eqref{eq-2.1}--\eqref{eq-2.2}, namely
\begin{align*}
	p(z)=\frac{1+(\overline{\tau_2}\tau_3+\overline{\tau_1}\tau_2 +\tau_1)z +(\overline{\tau_1}\tau_3+ \tau_1\overline{\tau_2}\tau_3 +\tau_2)z^2 +\tau_3 z^3}{1 +(\overline{\tau_2}\tau_3+ \overline{\tau_1}\tau_2 -\tau_1)z +(\overline{\tau_1}\tau_3- \tau_1\overline{\tau_2}\tau_3 -\tau_2)z^2 -\tau_3 z^3}, \;\;z\in\mathbb{D}.
\end{align*}
\end{lem}

\begin{lem}\cite{Cho-Kim-Sugawa-JMSJ-2007}\label{lem-2.2}
Let $A, B, C$ be real numbers and
\begin{align*}
	Y(A,B,C):=\max\{|A+ Bz +Cz^2| +1-|z|^2: z\in\overline{\mathbb{D}}\}.
\end{align*}
\noindent{(i)} If $AC\geq 0$, then
\begin{align*}
	Y(A,B,C)=\begin{cases}
		|A|+|B|+|C|, \;\;\;\;\;\;\;\;\;\;\;\;\;|B|\geq 2(1-|C|), \vspace{2mm}\\ 1+|A|+\dfrac{B^2}{4(1-|C|)}, \;\;\;\;\;|B|< 2(1-|C|).
	\end{cases}
\end{align*}
\noindent{(ii)} If $AC<0$, then
\begin{align*}
	Y(A,B,C)=\begin{cases}
		1-|A|+\dfrac{B^2}{4(1-|C|)}, \;\;\;\;-4AC(C^{-2}-1)\leq B^2\land|B|< 2(1-|C|),\vspace{2mm} \\ 1+|A|+\dfrac{B^2}{4(1+|C|)}, \;\;\;\; B^2<\min\{4(1+|C|)^2,-4AC(C^{-2}-1)\}, \vspace{2mm} \\ R(A,B,C), \;\;\;\;\;\;\;\;\;\;\;\;\;\;\;\;\;\;\; otherwise,
	\end{cases}
\end{align*}
where
\begin{align*}
	R(A,B,C):= \begin{cases}
		|A|+|B|-|C|, \;\;\;\;\;\;\;\;\;\;\;\;\; |C|(|B|+4|A|)\leq |AB|, \vspace{2mm}\\ -|A|+|B|+|C|, \;\;\;\;\;\;\;\;\;\;\; |AB|\leq |C|(|B|-4|A|), \vspace{2mm}\\ (|C| +|A|)\sqrt{1-\dfrac{B^2}{4AC}}, \;\;\; otherwise.
	\end{cases}
\end{align*}
\end{lem}

For a better clarity in our presentation, we divide this into two section consisting of different families of functions from the class $\mathcal{A}$ and prove our main results for starlike functions and convex functions associated with lune.
\section{\bf The second Hankel determinant of logarithmic coefficients of functions in the class $\mathcal{S}^{*}_{\mbox{\leftmoon}}$}
We obtain the following result finding the sharp bound of $ |H_{2,1}(F_{f}/2)| $ for functions in the class $ \mathcal{S}^{*}_{\mbox{\leftmoon}} $.
\begin{thm}\label{th-2.1}
Let $ f\in \mathcal{S}^{*}_{\mbox{\leftmoon}} $. Then
\begin{align}\label{Eq-3.1}
	|H_{2,1}(F_{f}/2)|\leq\frac{1}{16}.
\end{align}
The inequality is sharp for the function $ g\in \mathcal{S}^{*}_{\mbox{\leftmoon}} $ given by
\begin{align*}
	g(z)=z\exp\left(\int_{0}^{z}\frac{x^2 +\sqrt{1+x^4}-1}{x}dx\right)
	=z +\frac{z^3}{2} +\frac{z^5}{4}+\cdots.
\end{align*}
\end{thm}
\begin{proof}
Let $f\in\mathcal{S}^{*}_{\mbox{\leftmoon}}$. Then in view Definition \ref{def-1.1}, it follows that
\begin{align}\label{eq-3.1}
	\frac{zf^{\prime}(z)}{f(z)}=w(z) +\sqrt{1 +w^2(z)}, 
\end{align} 
where $w$ is a Schwarz function with $w(0)=0$ and $|w(z)|\leq 1$ in $\mathbb{D}$. Let $h\in\mathcal{P}$. Then we can write
\begin{align}\label{eq-3.2}
	w(z)=\frac{h(z)-1}{h(z)+1}.
\end{align}
From \eqref{eq-3.1} and \eqref{eq-3.2}, a simple computation shows that
\begin{align}\label{eq-3.3}
	\begin{cases}
		a_2=\dfrac{1}{2}c_1, \vspace{2mm}\\ a_3=\dfrac{1}{16}c^2_{1} + \dfrac{1}{4}c_2, \vspace{2mm}\\ a_4=\dfrac{1}{24}c_1 c_2 +\dfrac{1}{6} c_3 -\dfrac{1}{96}c^3_{1}.
	\end{cases}
\end{align}
Since the class $\mathcal{P}$ and $H_{2,1}(F_{f}/2)$ is invariant under rotation, we may assume that $c_1\in[0,2]$(see \cite{Carathéodory-MA-1907}; see also \cite[Vol. I, p. 80, Theorem 3]{Goodman-1983}), that is, in view of \eqref{eq-2.1}, that $\tau_1\in[0,1]$. Using \eqref{eq-1.4} and \eqref{eq-3.3}, it is easy to see that
\begin{align}\label{eq-3.4}
	H_{2,1}(F_{f}/2)&=\frac{1}{48} \left(a^4_2 - 12 a^2_3 + 12 a_2 a_4\right)\\&\nonumber=\frac{1}{3072}\left(-3c^4_{1} -8c^2_{1} c_2 -48c^2_{2} +64c_1 c_3\right).
\end{align}
By the Lemma \ref{lem-2.1} and \eqref{eq-3.4}, a straightforward computation shows that
\begin{align}\label{eq-3.5}
	H_{2,1}(F_{f}/2)&=\frac{1}{192}\left(-3\tau^4_{1} +4(1-\tau^2_{1})\tau^2_{1}\tau_{2}-4(1-\tau^2_{1})(3+\tau^2_{1})\tau^2_{2} \right.\\&\nonumber\left.\quad +16\tau_{1}\tau_{3}(1-\tau^2_{1})(1-|\tau^2_{2}|)\right).
\end{align}
Below, we discuss the following possible cases on $\tau_{1}$: \vspace{2mm}

\noindent{\bf Case 1.} Suppose that $\tau_{1}=1$. Then from \eqref{eq-3.5}, we easily obtain 
\begin{align*}
	|H_{2,1}(F_{f}/2)|=\frac{1}{64}.
\end{align*}

\noindent{\bf Case 2.} Let $\tau_{1}=0$. Then from \eqref{eq-3.5}, we see that 
\begin{align*}
	|H_{2,1}(F_{f}/2)|=\frac{1}{16}|\tau_{2}|^2\leq\frac{1}{16}.
\end{align*}

\noindent{\bf Case 3.} Suppose that $\tau_{1}\in(0,1)$. Applying the triangle inequality in \eqref{eq-3.5} and by using the fact that $|\tau_{3}|\leq 1$, we obtain
\begin{align}\label{eq-3.6}
	\nonumber|H_{2,1}(F_{f}/2)|&=\frac{1}{192}\left(|-3\tau^4_{1} +4(1-\tau^2_{1})\tau^2_{1}\tau_{2}-4(1-\tau^2_{1})(3+\tau^2_{1})\tau^2_{2}| \right.\\&\nonumber\left.\quad +16\tau_{1}(1-\tau^2_{1}) (1-|\tau^2_{2}|) \right)\\&=\frac{1}{12}\tau_{1}(1-\tau^2_{1})\left(|A+B\tau_{2}+C\tau^2_{2}| +1 -|\tau_{2}|^2\right),
\end{align}
where 
\begin{align*}
	A:=\frac{-3\tau^3_{1}}{16(1-\tau^2_{1})}, \;\; B:=\frac{\tau_{1}}{4}
	\;\; \mbox{and}\;\; C:=\frac{-(3+\tau^2_{1})}{4\tau_{1}}.
\end{align*}
Observe that $AC>0$, so we can apply case (i) of Lemma \ref{lem-2.2}. Now we check all the conditions of case (i). \vspace{2mm}

\noindent{\bf 3(a)} We note the inequality
\begin{align*}
	|B|-2(1-|C|)&=\frac{\tau_{1}}{4} -2\left(1-\frac{(3+\tau^2_{1})} {4\tau_{1}}\right)\\&=\frac{3\tau^2_{1} -8\tau_{1} +6}{4\tau_{1}}> 0
\end{align*}
which is true for all $\tau_{1}\in(0,1)$. Thus it follows from  Lemma \ref{lem-2.2} and the inequality \eqref{eq-3.6} that
\begin{align*}
	|H_{2,1}(F_{f}/2)|&\leq\frac{1}{12}\tau_{1}(1-\tau^2_{1})\left(|A|+|B|+|C|\right)\\&=\frac{1}{192}\left(12-4\tau^2_{1}-5\tau^4_{1}\right)\\&\leq\frac{1}{16}.
\end{align*}
\noindent{\bf 3(b)} Next, it is easy to check that
\begin{align*}
	|B|-2(1-|C|)&=\frac{\tau_{1}}{4} -2\left(1-\frac{(3+\tau^2_{1})} {4\tau_{1}}\right)\\&=\frac{3\tau^2_{1} -8\tau_{1} +6}{4\tau_{1}}< 0
\end{align*}
which is not true for all $\tau_{1}\in(0,1)$.\vspace{1.2mm}

Summarizing cases $1$, $2$, and $3$, the inequality \eqref{Eq-3.1} is established.\vspace{2mm}

To complete the proof, it is sufficient to show that the bound is sharp. In order to show that we consider the function $g\in\mathcal{S}^{*}_{\mbox{\leftmoon}}$ as follows
\begin{align*}
	g(z)=z\exp\left(\int_{0}^{z}\frac{x^2 +\sqrt{1+x^4}-1}{x}dx\right)
	=z +\frac{z^3}{2} +\frac{z^5}{4}+\cdots,
\end{align*}
with $a_2=a_4=0$ and $a_3=1/2$. By a simple computation, it can be easily shown that $|H_{2,1}(F_{g}/2)|=1/16$. This completes the proof.
\end{proof}

\section{\bf The second Hankel determinant of logarithmic coefficients of functions in the class $\mathcal{C}_{\mbox{\leftmoon}}$}
We obtain the following sharp bound of $ |H_{2,1}(F_{f}/2)| $ for functions in the class $ \mathcal{C}_{\mbox{\leftmoon}} $. 
\begin{thm}\label{th-2.3}
Let $ f\in \mathcal{C}_{\mbox{\leftmoon}} $. Then
\begin{align}\label{Eq-4.1}
	|H_{2,1}(F_{f}/2)|\leq\frac{23}{3264}
\end{align}
The inequality is sharp for the function $ h\in \mathcal{C}_{\mbox{\leftmoon}} $ is given by
\begin{align*}
	h(z)=\int_{0}^{z}\frac{h_{0}(x)}{x}dx=z +\frac{\sqrt{69}}{12\sqrt{17}}z^3 +\frac{1}{20}\left(\frac{69}{136}+\frac{\sqrt{69}}{4\sqrt{17}}\right)z^5 +\cdots,
\end{align*}
where $h_{0}(z)$ is given by \eqref{eq-4.7}.
\end{thm}
\begin{proof}
Let $ f\in \mathcal{C}_{\mbox{\leftmoon}} $. Then by the Definition \ref{def-1.1}, we see that
\begin{align}\label{eq-4.1}
	1+ \frac{zf^{\prime\prime}(z)}{f^{\prime}(z)}=w(z) +\sqrt{1 +w^2(z)}, 
\end{align} 
where $w$ is a Schwarz function with $w(0)=0$ and $|w(z)|\leq 1$ in $\mathbb{D}$. Let $h\in\mathcal{P}$. Then we can write
\begin{align}\label{eq-4.2}
	w(z)=\frac{h(z)-1}{h(z)+1}.
\end{align}
In view of \eqref{eq-4.1} and \eqref{eq-4.2}, a simple computation shows that
\begin{align}\label{eq-4.3}
\begin{cases}
a_2=\dfrac{1}{4}c_1, \vspace{2mm}\\ a_3=\dfrac{1}{48}c^2_{1} + \dfrac{1}{12}c_2, \vspace{2mm}\\ a_4=\dfrac{1}{96}c_1 c_2 +\dfrac{1}{24} c_3 -\dfrac{1}{384}c^3_{1}.
\end{cases}
\end{align}
Since the class $\mathcal{P}$ and $H_{2,1}(F_{f}/2)$ is invariant under rotation, we may assume that $c_1\in[0,2]$(see \cite{Carathéodory-MA-1907}; see also \cite[Volume I, page 80, Theorem 3]{Goodman-1983}), that is, in view of \eqref{eq-2.1}, that $\tau_1\in[0,1]$. From \eqref{eq-1.4} and \eqref{eq-4.3}, an easy computation leads to 
\begin{align}\label{eq-4.4}
	H_{2,1}(F_{f}/2)&=\dfrac{1}{48} \left(a^4_2 - 12 a^2_3 + 12 a_2 a_4\right)\\&\nonumber=\dfrac{1}{36864}\left(-7c^4_{1} -8c^2_{1} c_2 -64c^2_{2} +96c_1 c_3\right).
\end{align}
In view of Lemma \ref{lem-2.1} and \eqref{eq-4.4}, we obtain
\begin{align}\label{eq-4.5}
	H_{2,1}(F_{f}/2)&=\frac{1}{2304}\left(-3\tau^4_{1} +12(1-\tau^2_{1})\tau^2_{1}\tau_{2}-8(1-\tau^2_{1})(2+\tau^2_{1})\tau^2_{2} \right.\\&\nonumber\left.\quad +24\tau_{1}\tau_{3}(1-\tau^2_{1})(1-|\tau^2_{2}|)\right).
\end{align}
Now, we may have the following cases on $\tau_{1}$: \vspace{2mm}

\noindent{\bf Case 1.} Suppose that $\tau_{1}=1$. Then from \eqref{eq-4.5}, we obtain 
\begin{align*}
	|H_{2,1}(F_{f}/2)|=\frac{1}{768}.
\end{align*}

\noindent{\bf Case 2.} Let $\tau_{1}=0$. Then from \eqref{eq-3.5}, we have the following estimate 
\begin{align*}
	|H_{2,1}(F_{f}/2)|=\frac{1}{144}|\tau_{2}|^2\leq\frac{1}{144}.
\end{align*}

\noindent{\bf Case 3.} Suppose that $\tau_{1}\in(0,1)$. Applying the triangle inequality in \eqref{eq-4.5} and by using the fact that $|\tau_{3}|\leq 1$, we obtain
\begin{align}\label{eq-4.6}
	\nonumber H_{2,1}(F_{f}/2)&\leq\frac{1}{2304}\left(|-3\tau^4_{1} +12(1-\tau^2_{1})\tau^2_{1}\tau_{2}-8(1-\tau^2_{1})(2+\tau^2_{1})\tau^2_{2}| \right.\\&\nonumber\left.\quad +24\tau_{1}(1-\tau^2_{1})(1-|\tau^2_{2}|)\right)\\&=\frac{1}{96}\tau_{1}(1-\tau^2_{1})\left(|A+B\tau_{2}+C\tau^2_{2}| +1 -|\tau_{2}|^2\right),
\end{align}
where 
\begin{align*}
	A:=\frac{-\tau^3_{1}}{8(1-\tau^2_{1})}, \;\; B:=\frac{\tau_{1}}{2}
	\;\; \mbox{and}\;\; C:=\frac{-(2+\tau^2_{1})}{3\tau_{1}}.
\end{align*}
Observe that $AC>0$, so we can apply case (i) of Lemma \ref{lem-2.2}. Now we check all the conditions of case (i). \vspace{2mm}

\noindent{\bf 3(a)} We note the inequality
\begin{align*}
	|B|-2(1-|C|)&=\frac{\tau_{1}}{2} -2\left(1- \frac{(2+\tau^2_{1})}{3\tau_{1}}\right)\\&=\frac{7\tau^2_{1} -12\tau_{1} +8}{6\tau_{1}}> 0,
\end{align*}
which is true for all $\tau_{1}\in(0,1)$. It follows from  Lemma \ref{lem-2.2} and the inequality \eqref{eq-4.6} that
\begin{align*}
	|H_{2,1}(F_{f}/2)|&\leq\frac{1}{96}\tau_{1}(1-\tau^2_{1})\left(|A|+|B|+|C|\right)\\&=\frac{1}{2304}\left(16+4\tau^2_{1}-17\tau^4_{1}\right)\\&\leq\frac{23}{3264}.
\end{align*}
\noindent{\bf 3(b)} For the second inequality, we see that
\begin{align*}
	|B|-2(1-|C|)&=\frac{\tau_{1}}{2} -2\left(1- \frac{(2+\tau^2_{1})}{3\tau_{1}}\right)\\&=\frac{7\tau^2_{1} -12\tau_{1} +8}{6\tau_{1}}< 0,
\end{align*}
which is not true for all $\tau_{1}\in(0,1)$. Summarizing, all the cases discussed above, we obtain the inequality of the result.\vspace{1.2mm}

In order to show the bound is sharp, we consider the function $ h\in \mathcal{C}_{\mbox{\leftmoon}}, $ for 
\begin{align}\label{eq-4.7}
	h_{0}(z)&=z\exp\left(\sqrt{\frac{69}{68}}\int_{0}^{z}\frac{x^2 +\sqrt{1+x^4}-1}{x}dx\right)\\&\nonumber=z +\frac{\sqrt{69}}{4\sqrt{17}}z^3 +\frac{1}{4}\left(\frac{69}{136}+\frac{\sqrt{69}}{4\sqrt{17}}\right)z^5 +\cdots,.
\end{align}
Let
\begin{align*}
	h(z)=\int_{0}^{z}\frac{h_{0}(x)}{x}dx=z +\frac{\sqrt{69}}{12\sqrt{17}}z^3 +\frac{1}{20}\left(\frac{69}{136}+\frac{\sqrt{69}}{4\sqrt{17}}\right)z^5 +\cdots,
\end{align*}
with $a_2=a_4=0$ and $a_3=\sqrt{69}/{12\sqrt{17}}$. A simple computation shows that $|H_{2,1}(F_{h}/2)|=23/32640$, which demonstrates that the bound is sharp. This completes the proof.
\end{proof}

\noindent\textbf{Compliance of Ethical Standards:}\\

\noindent\textbf{Conflict of interest.} The authors declare that there is no conflict  of interest regarding the publication of this paper.\vspace{1.5mm}

\noindent\textbf{Data availability statement.}  Data sharing is not applicable to this article as no datasets were generated or analyzed during the current study.


\begin{thebibliography}{99}	
	
	\bibitem{Ali-Allu-PAMS-2018} {\sc M. F. Ali} and {\sc V. Allu}, On logarithmic coefficients of some close-to-convex functions, \textit{Proc. Amer. Math. Soc.} \textbf{146} (2018), 1131–1142. 
	
	\bibitem{Ali-Allu-Thomas-CRMCS-2018} {\sc M.F. Ali, V. Allu} and {\sc D. K. Thomas}, On the third logarithmic coefficients of close- to- convex functions, \textit{Curr. Res. Math. Comput. Sci. II}, Publisher UWM, Olsztyn, (2018) 271–278.
	
	\bibitem{Allu-Arora-Shaji-MJM-2023} {\sc V. Allu, V. Arora}, and {\sc A. Shaji}, On the Second Hankel Determinant of Logarithmic Coefficients for Certain Univalent Functions, \textit{Mediterr. J. Math.} \textbf{20} (2023), 81.
	
	\bibitem{Carathéodory-MA-1907} {\sc C. Carathéodory}, Über den Variabilitatsbereich der Koeffizienten von Potenzreihen, die gegebene Werte nicht annehmen, \textit{Math. Ann.} \textbf{64} (1907), 95–115.
	
	\bibitem{Cho-Kim-Sugawa-JMSJ-2007} {\sc N. E. Cho, Y. C. Kim} and {\sc T. Sugawa}, A general approach to the Fekete-Szego problem, \textit{J. Math. Soc. Japan.} \textbf{59} (2007), 707-727.
	
	\bibitem{Cho-Kowalczyk-kwon-Lecko-Sim-RACSAM-2020} {\sc N. E. Cho, B. Kowalczyk, O. S. Kwon, A. Lecko} and {\sc Y. J. Sim}, On the third logarithmic coefficient in some subclasses of close-to-convex functions, \textit{Rev. R. Acad. Cienc. Exactas Fís. Nat.(Esp.)} \textbf{114}, Art: 52, (2020), 1–14.
	
	\bibitem{Cho-Kumar-Kwon-Sim-JIA-2019} {\sc N. E. Cho, V. Kumar, O.S. Kwon} and {\sc Y.J. Sim}, Coefficient bounds for certain subclasses of starlike functions, \textit{J. Inequal. Appl.} \textbf{2019} (2019) 276.
	
	\bibitem{Duren-1983-NY}{\sc P. T. Duren}, Univalent Functions. \textit{Springer-Verlag}, New York Inc (1983).
	
	\bibitem{Efraimidis-JMAA-2016}{\sc I. Efraimidis}, A generalization of Livingston’s coefficient inequalities for functions
	with positive real part, \textit{J. Math. Anal. Appl.} \textbf{435}  (2016), 369–379.
	
	\bibitem{Gandhi-Ravichandran-AEJM-2017} {\sc S. Gandhi} and {\sc V. Ravichandran}, Starlike functions associated with a lune, \textit{Asian-Eur. J. Math.} \textbf{10} (2017) 1750064.
	
	\bibitem{Girela-AASF-2000} {\sc D. Girela}, Logarithmic coefficients of univalent functions, \textit{Ann. Acad. Sci. Fenn.} \textbf{25} (2000), 337–350. 
	
	\bibitem{Goodman-1983} {\sc A. W. Goodman}, Univalent Functions (Mariner, Tampa, FL, 1983).
	
	\bibitem{Kowalczyk-Lecko-BAMS-2022} {\sc B. Kowalczyk} and {\sc A. Lecko}, Second Hankel determinant of logarithmic coefficients of convex and starlike functions, \textit{Bull. Aust. Math. Soc.} \textbf{105} (2022), no. 3, 458–467.
	
	\bibitem{Kowalczyk-Lecko-BMMS-2022} {\sc B. Kowalczyk} and {\sc A. Lecko}, Second Hankel determinant of logarithmic coefficients of convex and starlike functions of order alpha, \textit{Bull. Malays. Math. Sci. Soc.} \textbf{45} (2022), no. 2, 727–740.
	
	\bibitem{Kowalczyk-Lecko-RACSAM-2023} {\sc B. Kowalczyk} and {\sc A. Lecko}, The second Hankel determinant of the logarithmic coefficients of strongly starlike and strongly convex functions,\textit{ Rev. Real Acad. Cienc. Exactas Fis. Nat. Ser. A-Mat.} \textbf{117}, 91 (2023). 
	
	\bibitem{Kowalczyk-Lecko-Sim-BAMS-2018} {\sc B. Kowalczyk, A. Lecko} and {\sc Y. J. Sim}, The sharp bound for the Hankel determinant of the third kind for convex functions, \textit{Bull. Aust. Math. Soc.} \textbf{97} (2018), 435–445.
	
	\bibitem{Libera-Zlotkiewicz-PAMS-1982} {\sc R. J. Libera} and {\sc E. J. Zlotkiewicz}, Early coefficients of the inverse of a regular convex function, \textit{Proc. Amer. Math. Soc.} \textbf{85} (1982), 225-230.
	
	\bibitem{Libera-Zlotkiewicz-PAMS-1983} {\sc R. J. Libera} and {\sc E. J. Zlotkiewicz}, Coefficient bounds for the inverse of a function with derivatives in $\mathcal{P}$, \textit{Proc. Amer. Math. Soc.} \textbf{87} (1983), 251-257.
	
	\bibitem{Milin-1977-ET}{\sc I. M. Milin}, Univalent Functions and Orthonormal Systems (Nauka, Moscow, 1971) (in Russian); English translation, Translations of Mathematical Monographs, 49 (\textit{American Mathematical Society}, Providence, RI, 1977).
	
	\bibitem{Ponnusamy-Sugawa-BDSM-2021} {\sc S.Ponnusamy} and {\sc T. Sugawa}, Sharp inequalities for logarithmic coefficients and
	their applications, \textit{Bull. Sci. Math.} \textbf{166} (2021), 23.
	
	\bibitem{Raina-Sharma-Sokół-JCMA-2018} {\sc R. K. Raina, P. Sharma} and {\sc J. Sokół}, Certain classes of analytic functions related to the crescent-shaped regions, \textit{J. Contemp. Math. Anal.} \textbf{53} (2018) 355–362.
	
	\bibitem{Raina-Sokół-HJMS-2015} {\sc R. K. Raina} and {\sc J. Sokół}, On coefficient for certain class of starlike functions, \textit{Hacet. J. Math. Stat.} \textbf{44}(6) (2015) 1427–1433.
	
	\bibitem{Raina-Sokół-MMN-2015} {\sc R. K. Raina} and {\sc J. Sokół}, Some coefficient properties relating to a certain class of starlike functions, \textit{Miskolc Math. Notes} \textbf{44}(6) (2015) 1427–1433.
	
	\bibitem{Raina-Sokół-CRMASCP-2015} {\sc R. K. Raina} and {\sc J. Sokół}, Some properties related to a certain class of starlike functions, \textit{C. R. Math. Acad. Sci. Paris} \textbf{353}(11) (2015) 973–978.
	
	\bibitem{Raza-Riza-Thomas-BAMS-2023} {\sc M. Raza, A. Riaz} and {\sc D. K. Thomas}, The third Hanekl determinant for inverse coefficients of convex functions, \textit{Bull. Aust. Math. Soc.} (2023), 1-7.
	
	\bibitem{Sakaguchi-JMSJ-1959} {\sc K. Sakaguchi}, On a certain univalent mapping, \textit{J. Math. Soc. Japan} \textbf{11} (1959), 72–75.
	
	\bibitem{Sharma-Raina-Sokół-AMP-2019} {\sc P. Sharma, R. K. Raina} and {\sc J. Sokół}, Certain Ma–Minda type classes of analytic functions associated with the crescent-shaped region, \textit{Anal. Math. Phys.} \textbf{9} (2019) 1–17.
	
	\bibitem{Sim-Lecko-Thomas-AMPA-2021} {\sc Y. J. Sim, A. Lecko}, and {\sc D. K. Thomas}, The second Hankel determinant for strongly convex and Ozaki close to convex functions, \textit{Ann. Mat. Pura Appl.} \textbf{200}(6) (2021) 2515-2533.
	
	\bibitem{Thomas-PAMS-2016} {\sc D. K. Thomas}, On logarithmic coefficients of close to convex functions, \textit{Proc. Amer. Math. Soc.} \textbf{144} (2016), 1681–1687.
	 
	\bibitem{Toeplitz-1907} {\sc O. Toeplitz}, Zur Transformation der Scharen bilinearer Formen von unendlichvielen Ver¨anderlichen.
	 Mathematischphysikalis- che, Klasse, Nachr. der Kgl. Gessellschaft derWissenschaften zu G¨ottingen, (1907), 110–115.
	 
	\bibitem{Ye-Lim-FCM-2016} {\sc K. Ye} and {\sc L. H. Lim}, Every matrix is a product of Toeplitz matrices,\textit{ Found. Comput. Math.} \textbf{16} (2016), no. 3, 577–598.
	
     \bibitem{Zaprawa-BSMM-2021} {\sc P. Zaprawa}, Initial logarithmic coefficients for functions starlike with respect to symmetric points, \textit{Bol. Soc. Mat. Mex.} \textbf{27} (2021) 62.
\end{thebibliography}
\end{document}